\documentclass[a4paper]{amsart}

\usepackage[utf8]{inputenc}
\usepackage[T1]{fontenc}
\usepackage[english]{babel}

\usepackage{amsfonts}
\usepackage{amsmath}
\usepackage[alphabetic]{amsrefs}
\usepackage{amssymb}
\usepackage{amstext}
\usepackage{amsthm}

\usepackage{geometry}
\usepackage{fullpage}

\usepackage[shortlabels]{enumitem}
\usepackage{graphicx}
\usepackage{hyperref}
\usepackage{listings}
\usepackage{mathrsfs}
\usepackage{mathtools}
\usepackage{multicol}
\usepackage{shuffle}

\usepackage{tikz}
\usepackage{tikz-cd}
\usetikzlibrary{arrows,backgrounds,chains,decorations.pathreplacing,patterns,shapes, calc, positioning}

\usepackage{xcolor}
\definecolor{purpleish}{RGB}{102,0,102}
\definecolor{pinkish}{RGB}{255,153,204}

\newcommand{\area}{\mathsf{area}}
\newcommand{\dinv}{\mathsf{dinv}}

\newcommand{\dcomp}{\mathsf{dcomp}}

\newcommand{\D}{\mathsf{D}} 
\newcommand{\W}{\mathsf{W}} 
\newcommand{\LD}{\mathsf{LD}} 

\newcommand{\bA}{\mathbb A}
\newcommand{\cN}{\mathcal N}
\newcommand{\cM}{\mathcal M}

\newcommand{\e}{\underline{e}}
\newcommand{\PPi}{\mathbf{\Pi}}

\makeatletter

\pgfkeys{
	/tikz/sharp angle/.code={%
		\pgfsetarrowoptions{sharp >}{#1}%
		\pgfsetarrowoptions{sharp <}{-#1}%
	},
	/tikz/sharp > angle/.code={%
		\pgfsetarrowoptions{sharp >}{#1}%
	},
	/tikz/sharp < angle/.code={%
		\pgfsetarrowoptions{sharp <}{#1}%
	},
	/tikz/sharp protrude/.code=\csname if#1\endcsname\qrr@tikz@sharp@z@-0.05\p@\else\qrr@tikz@sharp@z@\z@\fi,
	/tikz/sharp protrude/.default=true
}

\newdimen\qrr@tikz@sharp@z@
\qrr@tikz@sharp@z@\z@
\pgfarrowsdeclare{sharp >}{sharp >}{%
	\edef\pgf@marshal{\noexpand\pgfutil@in@{and}{\pgfgetarrowoptions{sharp >}}}%
	\pgf@marshal
	\ifpgfutil@in@
	\edef\pgf@tempa{\pgfgetarrowoptions{sharp >}}
	\expandafter\qrr@tikz@sharp@parse\pgf@tempa\@qrr@tikz@sharp@parse
	\else
	\qrr@tikz@sharp@parse\pgfgetarrowoptions{sharp >}and-\pgfgetarrowoptions{sharp >}\@qrr@tikz@sharp@parse
	\fi
	\pgfmathparse{max(\pgf@tempa,\pgf@tempb,0)}%
	\let\qrr@tikz@sharp@max\pgfmathresult
	\pgfmathsetlength\pgf@xa{.5*\pgflinewidth * tan(\qrr@tikz@sharp@max)}%
	\pgfarrowsleftextend{+\pgf@xa}%
	\pgfarrowsrightextend{+\pgf@xa}%
}{%
	\edef\pgf@marshal{\noexpand\pgfutil@in@{and}{\pgfgetarrowoptions{sharp >}}}%
	\pgf@marshal
	\ifpgfutil@in@
	\edef\pgf@tempa{\pgfgetarrowoptions{sharp >}}
	\expandafter\qrr@tikz@sharp@parse\pgf@tempa\@qrr@tikz@sharp@parse
	\else
	\qrr@tikz@sharp@parse\pgfgetarrowoptions{sharp >}and-\pgfgetarrowoptions{sharp >}\@qrr@tikz@sharp@parse
	\fi
	\pgfmathsetlength\pgf@ya{.5*\pgflinewidth * tan(max(\pgf@tempa,\pgf@tempb,0))}%
	\pgfmathsetlength\pgf@xa{-.5*\pgflinewidth * tan(\pgf@tempa)}%
	\pgfmathsetlength\pgf@xb{-.5*\pgflinewidth * tan(\pgf@tempb)}%
	\advance\pgf@xa\pgf@ya
	\advance\pgf@xb\pgf@ya
	\ifdim\pgf@xa>\pgf@xb
	\pgftransformyscale{-1}%
	\pgf@xc\pgf@xb
	\pgf@xb\pgf@xa
	\pgf@xa\pgf@xc
	\fi
	\pgfpathmoveto{\pgfqpoint{\qrr@tikz@sharp@z@}{.5\pgflinewidth}}%
	\pgfpathlineto{\pgfqpoint{\pgf@xa}{.5\pgflinewidth}}%
	\pgfpathlineto{\pgfqpoint{\pgf@ya}{+0pt}}%
	\pgfpathlineto{\pgfqpoint{\pgf@xb}{-.5\pgflinewidth}}%
	\pgfpathlineto{\pgfqpoint{\qrr@tikz@sharp@z@}{-.5\pgflinewidth}}%
	\pgfusepathqfill
}
\pgfarrowsdeclare{sharp <}{sharp <}{%
	\edef\pgf@marshal{\noexpand\pgfutil@in@{and}{\pgfgetarrowoptions{sharp <}}}%
	\pgf@marshal
	\ifpgfutil@in@
	\edef\pgf@tempa{\pgfgetarrowoptions{sharp <}}
	\expandafter\qrr@tikz@sharp@parse\pgf@tempa\@qrr@tikz@sharp@parse
	\else
	\expandafter\qrr@tikz@sharp@parse\pgfgetarrowoptions{sharp <}and-\pgfgetarrowoptions{sharp <}\@qrr@tikz@sharp@parse
	\fi
	\pgfmathparse{max(\pgf@tempa,\pgf@tempb,0)}%
	\let\qrr@tikz@sharp@max\pgfmathresult
	\pgfmathsetlength\pgf@xa{.5*\pgflinewidth * tan(\qrr@tikz@sharp@max)}%
	\pgfarrowsleftextend{+\pgf@xa}%
	\pgfarrowsrightextend{+\pgf@xa}%
}{%
	\edef\pgf@marshal{\noexpand\pgfutil@in@{and}{\pgfgetarrowoptions{sharp <}}}%
	\pgf@marshal
	\ifpgfutil@in@
	\edef\pgf@tempa{\pgfgetarrowoptions{sharp <}}
	\expandafter\qrr@tikz@sharp@parse\pgf@tempa\@qrr@tikz@sharp@parse
	\else
	\expandafter\qrr@tikz@sharp@parse\pgfgetarrowoptions{sharp <}and-\pgfgetarrowoptions{sharp <}\@qrr@tikz@sharp@parse
	\fi
	\pgfmathsetlength\pgf@ya{.5*\pgflinewidth * tan(max(\pgf@tempa,\pgf@tempb,0))}%
	\pgfmathsetlength\pgf@xa{-.5*\pgflinewidth * tan(\pgf@tempa)}%
	\pgfmathsetlength\pgf@xb{-.5*\pgflinewidth * tan(\pgf@tempb)}%
	\advance\pgf@xa\pgf@ya
	\advance\pgf@xb\pgf@ya
	\ifdim\pgf@xa>\pgf@xb
	\pgftransformyscale{-1}%
	\pgf@xc\pgf@xb
	\pgf@xb\pgf@xa
	\pgf@xa\pgf@xc
	\fi
	\pgfpathmoveto{\pgfqpoint{\qrr@tikz@sharp@z@}{.5\pgflinewidth}}%
	\pgfpathlineto{\pgfqpoint{\pgf@xa}{.5\pgflinewidth}}%
	\pgfpathlineto{\pgfqpoint{\pgf@ya}{+0pt}}%
	\pgfpathlineto{\pgfqpoint{\pgf@xb}{-.5\pgflinewidth}}%
	\pgfpathlineto{\pgfqpoint{\qrr@tikz@sharp@z@}{-.5\pgflinewidth}}%
	\pgfusepathqfill
}
\def\qrr@tikz@sharp@parse#1and#2\@qrr@tikz@sharp@parse{\def\pgf@tempa{#1}\def\pgf@tempb{#2}}

\makeatother

\newcommand\multiset[2]%
{\mathchoice{\left(\kern-0.4em{\binom{#1}{#2}}\kern-0.4em\right)}
	{\bigl(\kern-0.2em{\binom{#1}{#2}}\kern-0.2em\bigr)}
	{\bigl(\kern-0.2em{\binom{#1}{#2}}\kern-0.2em\bigr)}
	{\bigl(\kern-0.2em{\binom{#1}{#2}}\kern-0.2em\bigr)}}

\let\existstemp\exists \renewcommand*{\exists}{\mathop \existstemp}
\let\foralltemp\forall \renewcommand*{\forall}{\mathop \foralltemp}

\def\quotient#1#2{\raise1ex\hbox{$#1$}\Big/\lower1ex\hbox{$#2$}}

\newcommand{\<}{\langle}
\renewcommand{\>}{\rangle}

\newtheorem{theorem}{Theorem}[section]

\newtheorem{proposition}[theorem]{Proposition}
\newtheorem{corollary}[theorem]{Corollary}
\newtheorem{conjecture}[theorem]{Conjecture}

\theoremstyle{definition}
\newtheorem{definition}[theorem]{Definition}

\theoremstyle{remark}
\newtheorem{remark}[theorem]{Remark}

\title{A proof of the compositional Delta conjecture}

\author{Michele D'Adderio}
\address{Universit\'e Libre de Bruxelles (ULB)\\D\'epartement de Math\'ematique\\ Boulevard du Triomphe, B-1050 Bruxelles\\ Belgium}\email{mdadderi@ulb.ac.be}

\author{Anton Mellit}
\address{University of Vienna, Oskar-Morgenstern-Platz 1, Vienna 1090, Austria}\email{anton.mellit@univie.ac.at}

\begin{document}
	
\begin{abstract}
	We prove a compositional refinement of the Delta conjecture (rise version) of Haglund, Remmel and Wilson \cite{Haglund-Remmel-Wilson-2015} for $\Delta_{e_{n-k-1}}'e_n$ which was stated in \cite{dadderio2019theta} in terms of Theta operators. 
\end{abstract}
	
\maketitle
\tableofcontents

\section{Introduction}

In \cite{Haglund-Remmel-Wilson-2015} Haglund, Remmel and Wilson formulated the \emph{Delta conjecture (rise version)}, which can be stated as
\[\Delta_{e_{n-k-1}}'e_n=\sum_{P\in \LD(n)^{\ast k}} q^{\dinv(P)} t^{\area(P)}x^P, \]
where the sum is over labelled Dyck paths of size $n$ with positive labels and $k$ decorated rises (see Sections~\ref{sec:comb_defs} for definitions). It turns out that for $k=0$ this formula reduces to the shuffle conjecture, recently proved in \cite{Carlsson-Mellit-ShuffleConj-2015}: see \cite{Willigenburg_History_Shuffle} for a nice exposition of this interesting story. The Delta conjecture (rise version) already attracted quite a bit of interest, and several of its consequences have been proved: e.g.\ see \cite{Garsia-Haglund-Remmel-Yoo-2017,Romero-Deltaq1-2017,Haglund-Rhoades-Shimozono-Advances,DAdderio_Iraci_VWyngaerd_t0,DAdderio-Iraci-VandenWyngaerd-GenDeltaSchroeder,DAdderio-Iraci-VandenWyngaerd-Delta-Square}, and \cite[Section~2]{TheBible} for a short survey of partial progress on the problem.

In \cite{dadderio2019theta} a new family of operators on symmetric functions has been introduced, the so called Theta operators, which allowed the authors to conjecture a compositional refinement of the Delta conjecture, which can be stated as
\begin{equation} \label{eq:intro_compDelta}
(-1)^{|\alpha|}\Theta_{e_k}\nabla C_{\alpha}=\mathop{\sum_{P\in \LD(n)^{\ast k}}}_{\dcomp(P)=\alpha}q^{\dinv(P)}t^{\area(P)} x^P,
\end{equation}	
where $\dcomp(P)$ is the diagonal composition determined by the points where the Dyck path of $P$ touches the main diagonal and the positions of the decorated rises (see Sections~\ref{sec:comb_defs} for definitions).

This conjecture at $k=0$ gives the compositional shuffle conjecture stated in \cite{Haglund-Morse-Zabrocki-2012}, which is precisely what has been proved in \cite{Carlsson-Mellit-ShuffleConj-2015}.

In this work we prove \eqref{eq:intro_compDelta}, getting the Delta conjecture as an immediate corollary.
\begin{remark}
In \cite{Haglund-Remmel-Wilson-2015} there is also a \emph{valley version} of the Delta conjecture, which is left open.
\end{remark}

\medskip

The rest of this paper is organized in the following way. In Sections~2 and 3 we introduce the notions and tools needed to state in Section~3 the compositional Delta conjecture. In Section~4 we recall some definitions about the Dyck path algebra introduced in \cite{Carlsson-Mellit-ShuffleConj-2015} and how the Delta conjecture has been reduced in \cite{dadderio2019theta} to an identity of operators on symmetric functions. In Section~5 we finally prove this operator identity.

\section*{Acknowledgements}
The authors are grateful to Erik Carlsson, Adriano Garsia, Jim Haglund, Sasha Iraci, Marino Romero and Anna Vanden Wyngaerd for interesting discussions. In particular we find Adriano Garsia's continuous encouragement invaluable for the whole subject.

The first author's work is supported by the Fonds Thelam project J1150080.

The second author's work is supported by the projects Y963-N35 and
P-31705 of the Austrian Science Fund.

\section{Symmetric functions} \label{sec:SF}

In this section we limit ourselves to introduce the necessary notation to state our main theorem. 

The main references that we will use for symmetric functions
are \cite{Macdonald-Book-1995}, \cite{Stanley-Book-1999} and \cite{Haglund-Book-2008}. 

\medskip

The standard bases of the symmetric functions that will appear in our
calculations are the complete $\{h_{\lambda}\}_{\lambda}$, elementary $\{e_{\lambda}\}_{\lambda}$, power $\{p_{\lambda}\}_{\lambda}$ and Schur $\{s_{\lambda}\}_{\lambda}$ bases.

\medskip

\emph{We will use the usual convention that $e_0=h_0=1$ and $e_k=h_k=0$ for $k<0$.}

\medskip

The ring $\Lambda$ of symmetric functions can be thought of as the polynomial ring in the power
sum generators $p_1, p_2, p_3,\dots$. This ring has a grading $\Lambda=\bigoplus_{n\geq 0}\Lambda^{(n)}$ given by assigning degree $i$ to $p_i$ for all $i\geq 1$. As we are working with Macdonald symmetric functions
involving two parameters $q$ and $t$, we will consider this polynomial ring over the field $\mathbb{Q}(q,t)$.
We will make extensive use of the \emph{plethystic notation}.

With this notation we will be able to add and subtract alphabets, which will be represented as sums of monomials $X = x_1 + x_2 + x_3+\cdots $. Then, given a symmetric function $f$, and thinking of it as an element of $\Lambda$, we denote by $f[X]$ the expression $f$ with $p_k$ replaced by $x_{1}^{k}+x_{2}^{k}+x_{3}^{k}+\cdots$, for all $k$. More generally, given any expression $Q(z_1,z_2,\dots)$, we define the plethystic substitution $f[Q(z_1,z_2,\dots)]$ to be $f$ with $p_k$ replaced by $Q(z_1^k,z_2^k,\dots)$.

We denote by $\<\, , \>$ the \emph{Hall scalar product} on symmetric functions, which can be defined by saying that the Schur functions form an orthonormal basis. We denote by $\omega$ the fundamental algebraic involution which sends $e_k$ to $h_k$, $s_{\lambda}$ to $s_{\lambda'}$ and $p_k$ to $(-1)^{k-1}p_k$.

With the symbol ``$\perp$'' we denote the operation of taking the adjoint of an operator with respect to the Hall scalar product, i.e.
\begin{equation}
\langle f^\perp g,h\rangle=\langle g,fh\rangle\quad \text{ for all }f,g,h\in \Lambda.
\end{equation}

For a partition $\mu\vdash n$, we denote by
\begin{equation}
\widetilde{H}_{\mu} \coloneqq \widetilde{H}_{\mu}[X]=\widetilde{H}_{\mu}[X;q,t]=\sum_{\lambda\vdash n}\widetilde{K}_{\lambda \mu}(q,t)s_{\lambda}
\end{equation}
the \emph{(modified) Macdonald polynomials}, where 
\begin{equation}
\widetilde{K}_{\lambda \mu} \coloneqq \widetilde{K}_{\lambda \mu}(q,t)=K_{\lambda \mu}(q,1/t)t^{n(\mu)}\quad \text{ with }\quad n(\mu)=\sum_{i\geq 1}\mu_i(i-1)
\end{equation}
are the \emph{(modified) Kostka coefficients} (see \cite[Chapter~2]{Haglund-Book-2008} for more details). 

The set $\{\widetilde{H}_{\mu}[X;q,t]\}_{\mu}$ is a basis of the ring of symmetric functions $\Lambda$ with coefficients in $\mathbb{Q}(q,t)$. This is a modification of the basis introduced by Macdonald \cite{Macdonald-Book-1995}, and they are the Frobenius characteristic of the so called Garsia-Haiman modules (see \cite{Garsia-Haiman-PNAS-1993}).

If we identify the partition $\mu$ with its Ferrers diagram, i.e. with the collection of cells $\{(i,j)\mid 1\leq i\leq \mu_j, 1\leq j\leq \ell(\mu)\}$, then for each cell $c\in \mu$ we refer to the \emph{arm}, \emph{leg}, \emph{co-arm} and \emph{co-leg} (denoted respectively as $a_\mu(c), l_\mu(c), a_\mu(c)', l_\mu(c)'$) as the number of cells in $\mu$ that are strictly to the right, above, to the left and below $c$ in $\mu$, respectively (see Figure~\ref{fig:notation}).

\begin{figure}[h]
	\centering
	\begin{tikzpicture}[scale=.4]
	\draw[gray,opacity=.4](0,0) grid (15,10);
	\fill[white] (1,10)|-(3,9)|- (5,7)|-(9,5)|-(13,2)--(15.2,2)|-(1,10.2);
	\draw[gray]  (1,10)|-(3,9)|- (5,7)|-(9,5)|-(13,2)--(15,2)--(15,0)-|(0,10)--(1,10);
	\fill[blue, opacity=.2] (0,3) rectangle (9,4) (4,0) rectangle (5,7); 
	\fill[blue, opacity=.5] (4,3) rectangle (5,4);
	\draw (7,4.5) node {\tiny{Arm}} (3.25,5.5) node {\tiny{Leg}} (6.25, 1.5) node {\tiny{Co-leg}} (2,2.5) node {\tiny{Co-arm}} ;
	\end{tikzpicture}
	\caption{}
	\label{fig:notation}
\end{figure}

We set
\begin{equation}
M \coloneqq (1-q)(1-t),
\end{equation}
and we define for every partition $\mu$
\begin{align}
B_{\mu} &  \coloneqq B_{\mu}(q,t)=\sum_{c\in \mu}q^{a_{\mu}'(c)}t^{l_{\mu}'(c)}\\
T_{\mu} &  \coloneqq T_{\mu}(q,t)=\prod_{c\in \mu}q^{a_{\mu}'(c)}t^{l_{\mu}'(c)}\\
\Pi_{\mu} &  \coloneqq \Pi_{\mu}(q,t)=\prod_{c\in \mu/(1)}(1-q^{a_{\mu}'(c)}t^{l_{\mu}'(c)})\quad (\mu\neq\varnothing).
\end{align}

Notice that
\begin{equation} \label{eq:Bmu_Tmu}
B_{\mu}=e_1[B_{\mu}]\quad \text{ and } \quad T_{\mu}=e_{|\mu|}[B_{\mu}].
\end{equation}

For every symmetric function $f[X]$ we set
\begin{equation}
f^*=f^*[X] \coloneqq f\left[\frac{X}{M}\right].
\end{equation}

The following linear operators were introduced in \cites{Bergeron-Garsia-ScienceFiction-1999,Bergeron-Garsia-Haiman-Tesler-Positivity-1999}, and they are at the basis of the conjectures relating symmetric function coefficients and $q,t$-combinatorics in this area. 

We define the \emph{nabla} operator on $\Lambda$ by
\begin{equation}
\nabla  \widetilde{H}_{\mu}=(-1)^{|\mu|}T_{\mu} \widetilde{H}_{\mu}\quad \text{ for all }\mu.
\end{equation}
Notice that traditionally there is no sign in the definition of nabla, but we follow here the convention in \cite{Garsia-Mellit-FiveTerm-2019}, as it makes it easier to state and use some results in that reference.

We define the \emph{Delta} operators $\Delta_f$ and $\Delta_f'$ on $\Lambda$ by
\begin{equation}
\Delta_f \widetilde{H}_{\mu}=f[B_{\mu}(q,t)]\widetilde{H}_{\mu}\quad \text{ and } \quad 
\Delta_f' \widetilde{H}_{\mu}=f[B_{\mu}(q,t)-1]\widetilde{H}_{\mu},\quad \text{ for all }\mu.
\end{equation}

Observe that on the vector space of symmetric functions homogeneous of degree $n$, denoted by $\Lambda^{(n)}$, the operator $\nabla$ equals $(-1)^n\Delta_{e_n}$. Moreover, for every $1\leq k\leq n$,
\begin{equation} \label{eq:deltaprime}
\Delta_{e_k}=\Delta_{e_k}'+\Delta_{e_{k-1}}'\quad \text{ on }\Lambda^{(n)},
\end{equation}
and for any $k>n$, $\Delta_{e_k}=\Delta_{e_{k-1}}'=0$ on $\Lambda^{(n)}$, so that $\Delta_{e_n}=\Delta_{e_{n-1}}'$ on $\Lambda^{(n)}$.

\medskip

In \cite{Haglund-Morse-Zabrocki-2012} the following operators were introduced: for any $m\geq 0$ and any $F[X]\in \Lambda$
\begin{equation}
\mathbb{C}_mF[X] \coloneqq (-1/q)^{m-1}\sum_{r\geq 0}q^{-r}h_{m+r}[X]h_{r}[X(1-q)]^\perp F[X],
\end{equation}
and for any composition $\alpha=(\alpha_1,\alpha_2,\dots,\alpha_l)$ of $n$, denoted $\alpha\vDash n$, we set
\begin{equation}
C_\alpha=C_\alpha[X;q] \coloneqq \mathbb{C}_{\alpha_1}\mathbb{C}_{\alpha_2}\cdots \mathbb{C}_{\alpha_l}(1).
\end{equation}

The symmetric functions $E_{n,k}$ were introduced in \cite{Garsia-Haglund-qtCatalan-2002} by means of the following expansion:
\begin{equation} \label{eq:def_Enk}
e_n\left[X\frac{1-z}{1-q}\right]=\sum_{k=1}^n \frac{(z;q)_k}{(q;q)_k}E_{n,k},
\end{equation}
where
\begin{equation}
(a;q)_n \coloneqq (1-a)(1-qa)(1-q^2a)\cdots (1-q^{n-1}a)
\end{equation}
is the standard notation for the $q$-\emph{rising factorial}.

Notice that setting $z=q$ we get
\begin{equation} \label{eq:en_sum_Enk}
e_n=E_{n,1}+E_{n,2}+\cdots +E_{n,n}.
\end{equation}

The following identity is proved in \cite{Haglund-Morse-Zabrocki-2012}*{Section~5}:
\begin{equation} \label{eq:Enk=sumCalpha}
E_{n,r}=\mathop{\sum_{\alpha\vDash n}}_{\ell(\alpha)=r}C_\alpha\quad \text{ for all }r=1,2,\dots, n,
\end{equation}
where $\ell(\alpha)$ denotes the length of the composition $\alpha$.

Together with \eqref{eq:en_sum_Enk} it gives immediately
\begin{equation} \label{eq:en_sum_Calpha}
e_n=\sum_{\alpha\vDash n} C_\alpha.
\end{equation}

\medskip

Recall the definition of the invertible linear operator $\mathbf{\Pi}$ on $\Lambda=\oplus_{n\geq 1}\Lambda^{(n)}$ defined as follows: for any non-empty partition $\mu$
\begin{equation}
\mathbf{\Pi} \widetilde{H}_\mu \coloneqq \Pi_\mu \widetilde{H}_\mu .
\end{equation}

For any symmetric function $f\in \Lambda^{(k)}$ we introduce the following \emph{Theta operators} on $\Lambda$: for every $F \in \Lambda^{(n)}$ we set
\begin{equation} \label{eq:def_Deltaf}
\Theta_fF  \coloneqq 
\left\{\begin{array}{ll}
\mathbf{\Pi}f^*\mathbf{\Pi}^{-1}F  & \text{if } n\geq 1\\
0  & \text{if } n=0 \text{ and } k\geq 1 \\
f\cdot F  & \text{if } n=0 \text{ and } k=0
\end{array} \right.
\end{equation}
It is clear that $\Theta_f$ is linear, and moreover, if $f$ is homogenous of degree $k$, then so is $\Theta_f$, i.e. \[\Theta_f\Lambda^{(n)}\subseteq \Lambda^{(n+k)} \qquad \text{ for }f\in \Lambda^{(k)}. \]

The following theorem is proved in \cite[Theorem~3.1]{dadderio2019theta}.
\begin{theorem} \label{thm:DeltakmGD}
	For $n\geq 1$ and $k\geq 0$,
	\begin{equation} \label{eq:thm:DeltakmGD}
	(-1)^{n-k}\Theta_{e_k}\nabla e_{n-k}=\Delta_{e_{n-k-1}}'e_n.
	\end{equation}
\end{theorem}
The following corollary immediately follows from Theorem~\ref{thm:DeltakmGD} and \eqref{eq:en_sum_Calpha}.

\begin{corollary} \label{cor:sumTouchingGD}
	For $n\geq 1$ and $k\geq 0$,
	\begin{equation} \label{eq:DeltaCalpha}
	(-1)^{n-k}\sum_{\alpha\vDash n-k}\Theta_{e_k}\nabla C_\alpha = \Delta_{e_{n-k-1}}'e_n. 
	\end{equation}
\end{corollary}

\section{Combinatorial definitions} \label{sec:comb_defs}

\begin{definition}
	A \emph{Dyck path} of size $n$ is a lattice paths going from $(0,0)$ to $(n,n)$ consisting of east or north unit steps, always staying weakly above the line $x=y$ called the \emph{main diagonal}. The set of Dyck paths is denoted by $\D(n)$.
\end{definition}

\begin{definition}
	A \emph{labelling} or \emph{word} of a Dyck path $\pi$ of size $n$ ending east  is an element $w\in \mathbb N^{n}$ such that when we label the $i$-th vertical step of $\pi$ with $w_i$ the labels appearing in each column of $\pi$ are strictly increasing from bottom to top (cf.\ Figure~\ref{fig:pldExample0}). The set of such labellings is denoted by $\W(\pi)$. 
	
	A \emph{labelled Dyck path} is an element $P=(\pi, w)$ of 
	\begin{align*}
	&\LD(n)\coloneqq \{(\pi, w) \mid \pi \in \D(n), w \in \W(\pi) \}.
	\end{align*}
\end{definition}

\begin{figure*}[!ht]
	\centering
	\begin{tikzpicture}[scale = .6]
	
	\draw[step=1.0, gray!60, thin] (0,0) grid (8,8);
	
	\draw[gray!60, thin] (0,0) -- (8,8);
	
	\draw[blue!60, line width=2pt] (0,0) -- (0,1) -- (0,2) -- (1,2) -- (2,2) -- (2,3) -- (2,4) -- (2,5) -- (3,5) -- (4,5) -- (4,6) -- (4,7) -- (4,8) -- (5,8) -- (6,8) -- (7,8) -- (8,8);
	
	\draw (0.5,0.5) circle (0.4 cm) node {$2$};
	\draw (0.5,1.5) circle (0.4 cm) node {$3$};
	\draw (2.5,2.5) circle (0.4 cm) node {$1$};
	\draw (2.5,3.5) circle (0.4 cm) node {$4$};
	\draw (2.5,4.5) circle (0.4 cm) node {$6$};
	\draw (4.5,5.5) circle (0.4 cm) node {$1$};
	\draw (4.5,6.5) circle (0.4 cm) node {$2$};
	\draw (4.5,7.5) circle (0.4 cm) node {$6$};
	\end{tikzpicture}
	\caption{Example of an element in $\LD(8)$ with reading word $21341626$.}
	\label{fig:pldExample0}
\end{figure*}
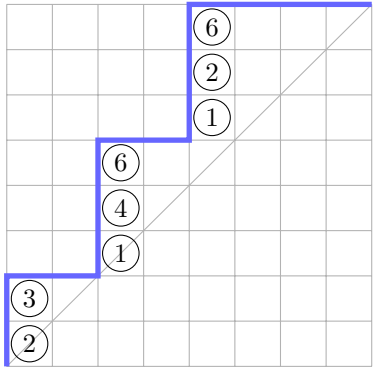

\begin{definition}
	Let $\pi$ be a Dyck path of size $n$. We define its \emph{area word} to be the sequence of integers $a(\pi) = (a_1(\pi),a_2(\pi), \cdots, a_{n}(\pi))$ such that the $i$-th vertical step of the path starts from the diagonal $y=x+a_i(\pi)$. For example the path in Figure~\ref{fig:pldExample0} has area word $(0, \, 1, \, 0, \, 1, \, 2, \, 1, \, 2, \, 3)$.
\end{definition}

\begin{definition}\label{def: monomial path}
	We define for each $P\in \LD(n)$ a monomial in the variables $x_1,x_2,\dots$: we set \[ x^P \coloneqq \prod_{i=1}^{n} x_{l_i(P)} \] where $l_i(P)$ is the label of the $i$-th vertical step of $P$ (the first being at the bottom). 
\end{definition}

\begin{definition}\label{def: rise}
	The \emph{rises} of a Dyck path $\pi$ are the indices \[ r(\pi) \coloneqq \{2\leq i \leq n\mid a_{i}(\pi)>a_{i-1}(\pi)\},\] or the vertical steps that are directly preceded by another vertical step. 
	
	A \emph{decorated Dyck path} is a pair $P=(\pi, dr)$ where $\pi$ is a Dyck path and $dr\subseteq r(\pi)$. We set 
	\[\D(n)^{\ast k} \coloneqq \{(\pi, dr)\mid \pi\in \D(n), dr\subseteq r(\pi), \vert dr\vert= k\}.\]
	A \emph{labelled decorated Dyck path} is a triple $(\pi, dr, w)$ where $\pi$ is a Dyck path, $dr\subseteq r(\pi)$ and $w$ is a labelling of $\pi$. We set 
	\[\LD(n)^{\ast k}\coloneqq \{(\pi, dr, w)\mid (\pi,w)\in \LD(n), dr\subseteq r(\pi), \vert dr\vert= k \}\]
	We will sometimes use the natural identification $\LD(n)^{\ast 0}= \LD(n)$.
\end{definition}

\begin{definition}\label{def: reading word}
	Given a labelled Dyck path $P$, we define its \emph{reading word} $\sigma(P)$ as the sequence of \emph{nonzero} labels, read starting from the main diagonal $y=x$ going bottom left to top right, then moving to the next diagonal, $y=x+1$ again going bottom left to top right, and so on.
	
	If the reading word of $P$ is $r_1\dots r_n$ then the  \emph{reverse reading word} of $P$ is $r_n\dots r_1$. 
\end{definition}
See Figure~\ref{fig:pldExample0} and Figure~\ref{fig:pldExample1} for an example.

\begin{figure*}[!ht]
	\centering
	\begin{tikzpicture}[scale = .6]
	
	\draw[step=1.0, gray!60, thin] (0,0) grid (8,8);
	
	\draw[gray!60, thin] (0,0) -- (8,8);
	
	\draw[blue!60, line width=2pt] (0,0) -- (0,1) -- (0,2) -- (1,2) -- (2,2) -- (2,3) -- (2,4) -- (2,5) -- (3,5) -- (4,5) -- (4,6) -- (4,7) -- (4,8) -- (5,8) -- (6,8) -- (7,8) -- (8,8);
	
	\draw (0.5,0.5) circle (0.4 cm) node {$2$};
	\draw (0.5,1.5) circle (0.4 cm) node {$3$};
	\draw (2.5,2.5) circle (0.4 cm) node {$1$};
	\draw (2.5,3.5) circle (0.4 cm) node {$4$};
	\draw (2.5,4.5) circle (0.4 cm) node {$6$};
	\draw (4.5,5.5) circle (0.4 cm) node {$1$};
	\draw (4.5,6.5) circle (0.4 cm) node {$2$};
	\draw (4.5,7.5) circle (0.4 cm) node {$6$};
	
	\node at (1.5,3.5) {$\ast$};
	\node at (3.5,6.5) {$\ast$};
	
	\end{tikzpicture}
	\caption{Example of an element in $\LD(6)^{\ast 2}$ with reading word $21314626$.}
	\label{fig:pldExample1}
\end{figure*}

We define two statistics on this set that reduce to the same statistics as defined in \cite{HHLRU-2005} when $k=0$. 

\begin{definition}
	\label{def:sqarea}
	Let $P=(\pi, dr) \in \D(n)^{\ast k}$. Define 
	\[
	\area(P) \coloneqq \sum_{i\not \in dr} a_i(\pi) .
	\] More visually, the area is the number of whole squares between the path and the main diagonal and not contained in rows containing a decorated rise. 
	
	If $P=(\pi, dr,w)\in \LD(n)^{\ast k}$ then we set $\area(P)=\area((\pi, dr))$. In other words, the area of a path does not depend on its labelling. 
\end{definition}

For example, the path in Figure~\ref{fig:pldExample1} has area $6$. 

\begin{definition} \label{def: dinv SQ}
	Let $P=(\pi, dr, w) \in \LD(n)^{\ast k}$. For $1 \leq i < j \leq n$, we say that the pair $(i,j)$ is an \emph{inversion} if
	\begin{itemize}
		\item either $a_i(\pi) = a_j(\pi)$ and $w_i < w_j$ (\emph{primary inversion}),
		\item or $a_i(\pi) = a_j(\pi) + 1$ and $w_i > w_j$ (\emph{secondary inversion}),
	\end{itemize}
	where $w_i$ denotes the $i$-th letter of $w$, i.e. the label of the vertical step in the $i$-th row.
	
	Then we define 
	\begin{align*}
	\dinv(P) & \coloneqq \# \{ 0\leq i < j \leq n \mid (i,j) \; \text{is an inversion}\}.
	\end{align*} 
\end{definition}

For example, the path in Figure~\ref{fig:pldExample1} has dinv $3$: $1$ primary inversion, i.e. $(2,4)$, and $2$ secondary inversions, i.e. $(2,3)$ and $(5,6)$.

\begin{definition}
	Given $P\in \LD(n)^{\ast k}$ we define its \emph{diagonal composition} $\dcomp(P)$ to be the composition of $n-k$ whose $i$-th part is the number of rows of $P$ without a decoration $\ast$ that lie between the $i$-th and the $(i+1)$-th vertical step of $P$ on the main diagonal (or from the $i$-th step onwards if it is the last such step). See Figure~\ref{fig:composition} for an example. If $P\in D(n)^{\ast k}$, its diagonal composition is defined identically.
\end{definition}

\begin{figure}[!ht]
	\centering
	\begin{tikzpicture}[scale = 0.6]
	\draw[gray!60, thin] (0,0) grid (12,12);
	\draw[gray!60, thin] (0,0) -- (12,12);
	
	\draw[blue!60, line width=1.6pt] (0,0) -- (0,1) -- (0,2) -- (1,2) -- (2,2) -- (2,3) -- (3,3) -- (3,4) -- (4,4) -- (4,5) -- (4,6) -- (4,7) -- (5,7) -- (6,7) -- (6,8) -- (7,8) -- (8,8) -- (8,9) -- (9,9) -- (9,10) -- (9,11) -- (10,11) -- (10,12) -- (11,12) -- (12,12);
	
	\draw
	(0.5,0.5) circle(0.4 cm) node {$2$}
	(0.5,1.5) circle(0.4 cm) node {$6$}
	(2.5,2.5) circle(0.4 cm) node {$3$}
	(3.5,3.5) circle(0.4 cm) node {$7$}
	(4.5,4.5) circle(0.4 cm) node {$2$}
	(4.5,5.5) circle(0.4 cm) node {$1$}
	(4.5,6.5) circle(0.4 cm) node {$4$}
	(6.5,7.5) circle(0.4 cm) node {$1$}
	(8.5,8.5) circle(0.4 cm) node {$8$}
	(9.5,9.5) circle(0.4 cm) node {$3$}
	(9.5,10.5) circle(0.4 cm) node {$5$}
	(10.5,11.5) circle(0.4 cm) node {$9$};
	
	\node at (-0.5,1.5) {$\ast$};
	\node at (3.5,5.5) {$\ast$};
	\end{tikzpicture}
	\caption{A partially labelled Dyck path with diagonal composition $\alpha = (1,1,1,3,1,3)$.}
	\label{fig:composition}
\end{figure}
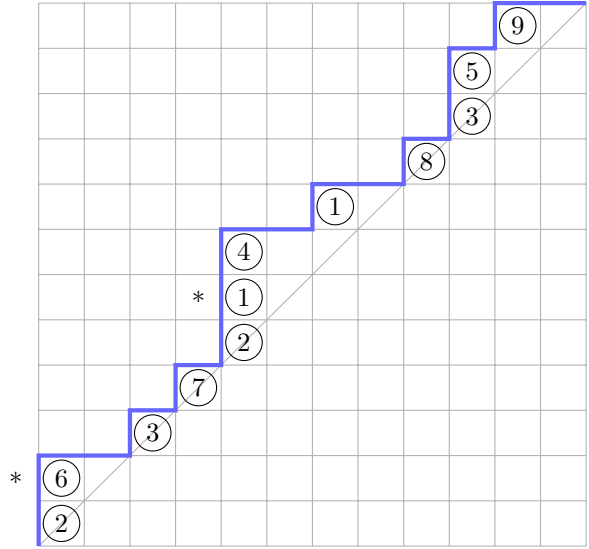

\section{Statements of Delta conjectures} \label{sec:statms}

In this section we state our refined conjectures. 

\medskip

The following conjecture is due to Haglund, Remmel and Wilson \cite{Haglund-Remmel-Wilson-2015}.

\begin{conjecture}[Delta (rise version)] \label{conj:GenDelta}
	Given $n,k\in \mathbb{N}$ with $n>k\geq 0$,
	\begin{equation}
	\Delta_{e_{n-k-1}}'e_n= \sum_{P\in \mathsf{LD}(n)^{\ast k}} q^{\dinv(P)}t^{\area(P)} x^P.
	\end{equation}	
\end{conjecture}

In \cite[Conjecture~5.4]{dadderio2019theta} it is stated a \emph{compositional} refinement of the \emph{Delta conjecture}, i.e.\ of the case $k=0$ of Conjecture~\ref{conj:GenDelta}.
\begin{conjecture}[Compositional Delta] \label{conj:compDelta}
	Given $n,k\in \mathbb{N}$, $n>k\geq 0$ and $\alpha\vDash n-k$,
	\begin{equation}\label{eq:compDelta}
	(-1)^{n-k}\Theta_{e_k}\nabla C_{\alpha}=\mathop{\sum_{P\in \LD(n)^{\ast k}}}_{\dcomp(P)=\alpha}q^{\dinv(P)}t^{\area(P)} x^P.
	\end{equation}	
\end{conjecture}
\begin{remark}
	We observe immediately that the compositional Delta conjecture implies the Delta conjecture: just sum \eqref{eq:compDelta} over $\alpha\vDash n-k$ and use \eqref{eq:DeltaCalpha}. 
\end{remark}

The main result of this article is a proof of these conjectures.

\section{Relation to the Dyck path algebra} \label{sec:reduction}

Following \cite{Carlsson-Mellit-ShuffleConj-2015}, we now introduce the operators of the \emph{Dyck path algebra} $\mathbb{A}=\mathbb{A}_q$.

Given a polynomial $P$ depending on variables $u,v$, define the operator $\Upsilon_{uv}$ as

\begin{align*}
(\Upsilon_{uv} P)(u,v) & \coloneqq \frac{(q-1)vP(u,v) + (v-qu)P(v,u)}{v-u}
\end{align*}

In \cite{Carlsson-Mellit-ShuffleConj-2015} this operator is called $\Delta_{uv}$, but we changed the notation in order to avoid confusion with the $\Delta_f$ operator defined on $\Lambda$.

\begin{definition}[\cite{Carlsson-Mellit-ShuffleConj-2015}*{Definition~4.2}]
	For $k \in \mathbb{N}$, define $V_k \coloneqq \Lambda[y_1, \dots, y_k]=\Lambda\otimes \mathbb{Q}[y_1,\dots,y_k]$. Let 
	\[T_i \coloneqq \Upsilon_{y_i y_{i+1}} \colon V_k \rightarrow V_k\text{ for }1 \leq i \leq k-1. \]
	
	We define the operators $d_+ \colon V_k \rightarrow V_{k+1}$ and $d_- \colon V_k \rightarrow V_{k-1}$: for $F[X]\in V_k$
	\begin{align*}
	(d_+ F)[X] & \coloneqq T_1 T_2 \cdots T_k (F[X + (q-1) y_{k+1}]) \\
	(d_- F)[X] & \coloneqq -F[X - (q-1)y_k] \sum_{i\geq 0} \left.  (-1/y_k )^{i}e_i[X]  \right|_{{y_k}^{-1}}.
	\end{align*}
\end{definition}

The following theorem is an immediate consequence of Theorem~6.22 and Equation~(41) in \cite{dadderio2019theta}.

\begin{theorem}\label{thm: comp-recursion}
	If $\alpha$ is a composition of length $\ell$, then we have 
	\begin{equation} \label{eq:thm_Dyck_path_rel}
	\mathop{\sum_{P\in \LD(n)^{\ast k}}}_{\dcomp(P)=\alpha}q^{\dinv(P)}t^{\area(P)} x^P = d_-^\ell M_\alpha^{\ast k}
	\end{equation}
	where $M_\alpha^{\ast k} \in V_\ell$ is defined by the recursive relations 
	\begin{equation} \label{eq:recursion_1}
	M_{(1)\alpha}^{\ast k} = d_+ M_\alpha^{\ast k} + \frac{1}{q-1} [d_-, d_+] M_{\alpha(1)}^{\ast k-1},
	\end{equation}
	and for $a > 1$
	\begin{equation} \label{eq:recursion_a}
	M_{(a)\alpha}^{\ast k} = \frac{t^{a-1}}{q-1} [d_-, d_+] \left( \sum_{\beta \vDash a-1} d_-^{\ell(\beta)-1} M_{\alpha \beta}^{\ast k} + \sum_{\beta \vDash a} d_-^{\ell(\beta)-1} M_{\alpha \beta}^{\ast k-1} \right) ,
	\end{equation} 
	with initial conditions $M_\varnothing^{\ast k} = \delta_{k,0}$.
\end{theorem}

It follows immediately from Theorem~\ref{thm: comp-recursion} that the following theorem, which is the main result of the present article, is equivalent to the compositional Delta conjecture \eqref{eq:compDelta}.
\begin{theorem}  \label{thm:operatDelta}
	If $\alpha$ is a composition of length $\ell$, then
	\begin{equation}
	(-1)^{|\alpha|}\Theta_{e_k}\nabla C_\alpha = d_-^\ell M_\alpha^{\ast k},
	\end{equation}
	with $M_\alpha^{\ast k}$ defined as in Theorem~\ref{thm: comp-recursion}.
\end{theorem}

The rest of this work is devoted to a proof of this theorem.

\section{Proof of Theorem \eqref{thm:operatDelta}} \label{sec:theproof}

We begin by proving the following identity.
\begin{remark}
Recently Conjecture~10.3 in \cite{dadderio2019theta}, which is another identity with Theta operators, has been proved  in \cite{romero2020proof} by a similar method. In fact Proposition~\ref{prop:commutation M} below can be used to obtain another proof of that conjecture.
\end{remark}

In writing identities it is convenient to use generating functions. Let $u, v$ be formal variables. Let 
\[
\Theta(u) \coloneqq \sum_{k=0}^\infty (-u)^k \Theta_{e_k},\quad \Delta(u) \coloneqq \sum_{k=0}^\infty (-u)^k \Delta_{e_k}, \quad \text{and}\quad 
\tau^*_u \coloneqq \sum_{k=0}^\infty (-u)^k \e_k^*,
\]
where the operator $\underline{e}_k^*$ is simply defined by $\e_k^*(f)\coloneqq e_k^*f$ for every symmetric function $f$.
\begin{proposition}\label{prop:theta nabla}
	We have the following identity:
	\begin{equation}\label{eq:theta nabla}
	\Theta(u) \nabla = \tau^*_u \nabla \tau^*_u.
	\end{equation}
	Alternatively, we  have
	\[
	\Theta_{e_k} \nabla = \sum_{i=0}^k \e_i^* \nabla \e_{k-i}^*.
	\]
\end{proposition}
\begin{proof}
	We quote the identity from \cite[Theorem 1.1, specialized to (4) and (2)] {Garsia-Mellit-FiveTerm-2019}:
	\begin{equation}\label{eq:five term}
	T_{1,0} T_{0,1} = T_{0,1} T_{1,1} T_{1,0}, \qquad 
	\end{equation}
	where 
	\[
	T_{1,0} = \sum_{k=0}^\infty (-u)^k \Delta_{e_k[-1/M+X]},\quad T_{0,1}=\tau^*_v,\quad T_{1,1}=\sum_{k=0}^\infty (uv)^k R_{k,k}
	\]
	for certain operators $R_{k,k}$. The series $T_{1,0}$ satisfies
	\[
	T_{1,0} = \sum_{k=0}^\infty (-u)^k e_k[-1/M] \cdot \Delta(u),
	\]
	and cancelling the series $(-u)^k e_k[-1/M]$ from both sides of \eqref{eq:five term} produces
	\begin{equation}\label{eq:gen series}
	\Delta(u) \tau^*_v = \tau^*_v \left(\sum_{k=0}^\infty (uv)^k R_{k,k}\right) \Delta(u).
	\end{equation}
	Suppose $f$ is a symmetric function of degree $d$. Collecting the coefficients of $u^m v^k$ in \eqref{eq:gen series} produces
	\[
	\Delta_{e_m} (e_k^* f) = \sum_{i=0}^{\min(m,k)} \e_{k-i}^* R_{i,i} \Delta_{e_{m-i}} (f).
	\]
	Let us specialize to $m=k+d$. Since the degree of $e_k^* f$ is $m$, the $\Delta$ operator on the left hand side can be replaced by $(-1)^{k+d} \nabla$. The terms on the right hand side vanish when $m-i>d$, which is equivalent to $i<k$. So only the term for $i=k$ on the right hand side survives and we obtain
	\[
	(-1)^{k+d} \nabla(e_k^* f) = R_{k,k} (-1)^d \nabla (f).
	\]
	So we have $R_{k,k} = (-1)^k \nabla \e_k^* \nabla^{-1}$ and  \eqref{eq:gen series} can be written as follows:
	\begin{equation}\label{eq:gen series 2}
	\Delta(u) \tau_v^* = \tau_v^* \nabla \tau_{uv}^* \nabla^{-1} \Delta(u).
	\end{equation}
	If we apply both sides to a function of positive degree, the result is divisible by $1-u$. In fact we have 
	\[
	\Delta(u) \tilde H_\mu = \prod_{c\in\mu} (1-u q^{a'_\mu(c)} t^{l'_\mu(c)}) \tilde H_\mu.
	\]
	 Diving by $1-u$ and specializing to $u=1$ produces the operator $\PPi$, so \eqref{eq:gen series 2} implies the following identity for $\PPi$:
	\[
	\PPi \tau_v^* (f) = \tau_v^* \nabla \tau_v^* \nabla^{-1} \PPi (f) \qquad (\deg f>0).
	\]
	From the definition of $\Theta(v)$ we have $\Theta(v)=\PPi \tau_v^* \PPi^{-1}$ when applied to functions of positive degree, so in this case the statement is equivalent to \eqref{eq:theta nabla}. On the other hand, applying \eqref{eq:gen series 2} to the function $1$ produces
	\[
	\Delta(u) (\tau_v^* (1)) = \tau_v^* \nabla \tau_{uv}^* (1).
	\]
	Setting $u=1$ in this identity produces $1 = \tau_v^* \nabla \tau_{v}^* (1)$. So \eqref{eq:theta nabla} is also true when applied to a function of degree $0$.
\end{proof}

The operation $\bar\omega$ on symmetric functions is defined by sending $F[X;q,t]$ to $F[-X;q^{-1},t^{-1}]$. The operators $\tau_u^*$ and $\nabla\bar\omega$ admit nice extensions to the space $V_*\coloneqq V_0\oplus V_1\oplus\cdots$, which is a module for the $\bA_{q,t}$-algebra (see \cite{Carlsson-Mellit-ShuffleConj-2015}).

Consider the algebra $\mathbb{A}^*=\mathbb{A}_{q^{-1}}$ with generators $d_{\pm}^*$, $T_i^*$, where $z_i$ denotes the image of the multiplication by $y_i$ under the isomorphism from $\mathbb{A}$ to $\mathbb{A}^*$ that sends generators to corresponding generators, and which is antilinear with respect to $q\mapsto q^{-1}$.

Recall (see \cite[Theorem~6.1]{Carlsson-Mellit-ShuffleConj-2015}) that $\mathbb{A}^*$ acts on $V_*$, and under this action
\[d_-^*=d_-,\quad T_i^*=T_i^{-1},\quad (d_+^*F)[X]=\gamma F[X+(q-1)y_{k+1}]\quad \text{for }F\in V_k, \]
where $\gamma$ is the operator that sends $y_i$ into $y_{i+1}$ for $i=1,2,\dots,k$ and $y_{k+1}$ to $ty_k$. Moreover (cf.\ \cite[Lemma~5.4]{Carlsson-Mellit-ShuffleConj-2015}) on $V_k$
\[ z_1= \frac{q^{k-1}}{q^{-1}-1}(d_+^*d_- -d_-d_+^*)T_{k-1}^{-1}\cdots T_1^{-1}. \]
\begin{proposition}\label{prop:operators on Aqt}
	The operator $\nabla\bar\omega$ extends to an antilinear operator $\cN$ on $V_\ast$ in such a way that we have
	\[
	\cN d_+ = d_+^* \cN, \quad \cN d_- = d_- \cN,\quad \cN T_i = T_i^{-1}\cN, \quad \cN y_i = z_i \cN, \quad \cN z_i = y_i \cN.
	\]
	The operator $\tau_u^*$ extends to a linear endomorphism of $V_k$ in such a way that it commutes with $d_-, T_i, d_+, y_i$ and we have
	\[
	d_+^* \tau_u^* = (1-u y_1) \tau_u^* d_+^*,\quad z_1 \tau_u^* = (1-u y_1) \tau_u^* z_1.
	\]
\end{proposition}
\begin{proof}
	The statement about $\nabla$ is \cite[Theorem 7.4]{Carlsson-Mellit-ShuffleConj-2015}. 
	The extension of $\tau_u^*$ is extracted from the proof of \cite[Proposition~3.13]{mellit2016toric}. Define $\tau_u^*$ on $V_k$ by
	\[
	\tau_u^* = \sum_{n=0}^\infty (-u)^n \e_n\left[\frac{X+(q-1)\sum_{i=1}^k y_i}{M}\right].
	\]
	 It is then straightforward to verify the statements about $\tau_u^*$.
\end{proof}

Now we combine Propositions \ref{prop:operators on Aqt} and \ref{prop:theta nabla} to obtain an extension of $\Theta(u)\nabla\bar\omega$ to $V_\ast$.

\begin{proposition}\label{prop:commutation M}
	The operator $\Theta(u)\nabla\bar\omega$ extends to an antilinear operator $\cM$ on $V_\ast$ in such a way that it commutes with $d_-$, $\cM T_i=T_i^{-1}\cM$, and we have
	\begin{equation}\label{eq:commutation M 1}
	\cM d_+ = (1-u y_1)^{-1} d_+^* \cM,\quad \cM y_1 = (1-u y_1)^{-1} z_1,
	\end{equation}
	\begin{equation}\label{eq:commutation M 2}
	\cM d_+^* = \left(1-(qt)^{-1} u (1-u y_1)^{-1} z_1\right) d_+ \cM,\quad \cM z_1 = \left(1-(qt)^{-1} u (1-u y_1)^{-1} z_1\right) y_1 \cM,
	\end{equation}
	\begin{equation}\label{eq:commutation M 3}
		d_+ \cM = \cM (1 - q t u y_1)^{-1} d_+^*, \quad y_1 \cM = \cM (1 - q t u y_1)^{-1} z_1.
	\end{equation}
\end{proposition}
\begin{proof}
	It is convenient to write $\Theta(u)\nabla\bar\omega$ as follows:
	\[
	\Theta(u)\nabla\bar\omega = \tau_u^* \nabla \tau_u^* \bar\omega = \tau_u^* \nabla  \bar\omega \left(\tau_{qt u}^*\right)^{-1} = \tau_u^* \cN \left(\tau_{qt u}^*\right)^{-1}.
	\]
	Then the statements are obtained by successively applying the commutation relations of Proposition \ref{prop:operators on Aqt}.
\end{proof}

For any composition $\alpha$ with $\alpha=(\alpha_1,\ldots,\alpha_\ell)$, $\sum_{i=1}^\ell \alpha_i=|\alpha|$ we have
\[
(-1)^{|\alpha|}C_\alpha = q^{\ell-|\alpha|} \bar\omega\left(d_-^\ell y_1^{\alpha_1-1} \cdots y_r^{\alpha_\ell-1} d_+^\ell (1)\right).
\]
Denote 
\[
y_\alpha = y_1^{\alpha_1-1} \cdots y_r^{\alpha_\ell-1} d_+^\ell (1) \in V_\ell.
\]
Define for any composition $\alpha$
\[
M_\alpha(u) = \sum_{k=0}^\infty (-u)^k M_\alpha^k :=  q^{\ell-|\alpha|} \cM(y_\alpha).
\]
\begin{theorem}
	We have $M_\alpha^k=M_\alpha^{* k}$ for all compositions $\alpha$ and all $k$. In particular, we have
	\[
	d_-^{\ell(\alpha)} M_\alpha^{* k} =(-1)^{|\alpha|} \Theta_{e_k} \nabla C_\alpha.
	\]
\end{theorem}
\begin{proof}
In order to show that $M_\alpha^k=M_\alpha^{* k}$ it is sufficient to verify that the $u$-coefficients of $M_\alpha(u)$ satisfy the recursion relations \eqref{eq:recursion_1}, \eqref{eq:recursion_a} of Theorem~\ref{thm: comp-recursion}. The main idea for verifying these relations is to start with the right hand side and simplify it to obtain the left hand side instead of the other way around. 

We begin with relation \eqref{eq:recursion_a}.
Consider the expression
\[
F_{a, \alpha} := (d_- d_+ - d_+ d_-) \sum_{\beta\vDash a-1} d_-^{\ell(\beta)-1} (M_{\alpha\beta}(u)) \qquad(a\geq 2).
\]
Using \eqref{eq:commutation M 3} it can be written as\footnote{Keep in mind that $\cM$ is antilinear, so $\cM(qf)=q^{-1}\cM(f)$, $\cM(tf)=t^{-1}\cM(f)$ for any expression $f$.}
\begin{equation}\label{eq:intermediate1}
	F_{a, \alpha} = q^{\ell(\alpha)-|\alpha|} \sum_{\beta\vDash a-1}  \cM\left( q^{a-1-\ell(\beta)} (1-qt u y_1)^{-1} (d_- d_+^* - d_+^* d_-) d_-^{\ell(\beta)-1} y_{\alpha\beta}\right).
\end{equation}
The following identity has been established in \cite[Proposition~6.6]{Carlsson-Mellit-ShuffleConj-2015}:
\begin{equation}\label{eq:y recursion}
y_{(a) \alpha} = \frac{t^{1-a}}{q-1}(d_+^* d_- - d_- d_+^*) \sum_{\beta\vDash a-1} q^{1-l(\beta)} d_-^{l(\beta)-1} (y_{\alpha\beta}) \qquad(a\geq 2).
\end{equation}
So \eqref{eq:intermediate1} can be simplified to
\[
F_{a, \alpha} = -q^{\ell(\alpha)-|\alpha|} \cM\left( q^{a-2} (q-1) t^{a-1} (1-qtu y_1)^{-1} y_{(a+1) \alpha}\right).
\]
From $y_{(a+1)\alpha}=y_1 y_{(a)\alpha}$ we obtain
\[
F_{a,\alpha} - u F_{a+1,\alpha} = - q^{\ell(\alpha)-|\alpha|} \cM\left( q^{a-2} (q-1) t^{a-1} y_{(a) \alpha}\right) = (q-1)t^{1-a} M_{(a) \alpha}(u).
\]
This is equivalent to \eqref{eq:recursion_a}.

Now consider \eqref{eq:recursion_1}. Write the generating series for the second term on the right hand side:
\[
\frac{1}{q-1}[d_-,d_+] M_{\alpha(1)} (u) = \frac{q^{\ell(\alpha)-|\alpha|}}{q-1} \cM\left( (1-qtu y_1)^{-1} [d_-,d_+^*] y_{\alpha(1)}\right).
\]
Using \eqref{eq:y recursion} in the special case $a=2$ we obtain
\[
=-\frac{q^{\ell(\alpha)-|\alpha|}}{q-1} \cM\left( (1-qtu y_1)^{-1} (q-1) t y_{(2)\alpha}\right) = q^{\ell(\alpha)-|\alpha|} \cM\left(q t y_1 (1-qtu y_1)^{-1} y_{(1)\alpha}\right).
\]
On the other hand, the generating series for the first term on the right hand side of \eqref{eq:recursion_1} equals
\[
d_+ M_{\alpha}(u) = q^{\ell(\alpha)-|\alpha|} \cM\left( (1-qtu y_1)^{-1} y_{(1)\alpha} \right).
\]
Combining the two terms we obtain
\[
d_+ M_{\alpha}(u) - u \frac{1}{q-1}[d_-,d_+] M_{\alpha(1)} (u) = q^{\ell(\alpha)-|\alpha|} \cM\left(y_{(1)\alpha} \right) = M_{(1)\alpha}(u).
\]
This is equivalent to \eqref{eq:recursion_1}.
\end{proof}


\bibliographystyle{amsalpha}
\bibliography{Biblebib}

\end{document}